\documentclass{amsart}
\usepackage{amssymb,amscd}
\usepackage{graphicx}

\usepackage{latexsym}
\usepackage{multirow}
\usepackage{setspace}

\usepackage{amsthm}
\newtheorem{theorem}{Theorem}

\newtheorem{lemma}{Lemma}
\newtheorem{cor}{Corollary}

\newcommand{\res}[1]{\begin{array}[d]{l}\\{\rm Res}\\^{#1}\end{array}\hspace{-1mm}}

\newcommand{\bc}{\mathbb{C}}

\newcommand{\bp}{\mathbb{P}}
\newcommand{\bq}{\mathbb{Q}}
\newcommand{\br}{\mathbb{R}}
\newcommand{\bz}{\mathbb{Z}}
\newcommand{\modm}{\mathcal{M}}

\newcommand{\f}{\mathcal{F}}

\newcommand{\fat}{\f\hspace{-.3mm}{\rm at}_{g,n}}

\begin{document}

\title[String and dilaton equations for counting lattice points]{String and dilaton equations for counting lattice points in the moduli space of curves.}
\author{Paul Norbury}
\address{Department of Mathematics and Statistics\\
University of Melbourne\\Australia 3010}
\email{pnorbury@ms.unimelb.edu.au}
\thanks{The author was partially supported by ARC Discovery project DP1094328.}
\keywords{}
\subjclass[2010]{32G15; 30F30; 05A15}
\date{}

\begin{abstract}

\noindent We prove that the Eynard-Orantin symplectic invariants of the curve $xy-y^2=1$ are the orbifold Euler characteristics of the moduli spaces of genus $g$ curves.  We do this by associating to the Eynard-Orantin invariants of $xy-y^2=1$ a problem of enumerating covers of the two-sphere branched over three points.  This viewpoint produces new recursion relations---string and dilaton equations---between the quasi-polynomials that enumerate such covers.

\end{abstract}

\maketitle

\section{Introduction}
Consider a genus zero plane curve $C\subset\bc^2$ such that the branch points of the first coordinate $x:C\to\bc$ are simple.  Eynard and Orantin \cite{EOrInv} have developed a sequence of invariants of such plane curves to study enumerative problems in geometry.  In this paper we describe a Hurwitz problem related to the Eynard-Orantin invariants of the plane curve $xy-y^2=1$.  The invariants $\omega^{g}_n(C)=\omega^{g}_n(z_1,...,z_n)dz_1...dz_n$ are multidifferentials on $C$, for all integers $g\geq 0$ and $n>0$,  that satisfy recursion relations with a Virasoro algebra structure.  The meromorphic differentials $\omega^{g}_1(C)$ are used to define {\em symplectic} invariants $F^{(g)}(C)\in\bc$ (essentially $\omega^{g}_0(C)$) which are invariant under automorphisms of $\bc^2$ that preserve the symplectic form $dx\wedge dy$.  See Section~\ref{sec:EO} for a precise definition of the invariants in the more general setting of Torelli marked curves of genus $g$ immersed in $\bc^2$.   For different choices of the curve $C$ the invariants $\omega^{g}_n(C)$ store enumerative information such as tautological intersection numbers over the moduli space of genus $g$ curves with $n$ labeled points \cite{EynMum,KonInt}, simple Hurwitz numbers \cite{BEMSMat,BMaHur}, Weil-Petersson volumes of the moduli space of curves \cite{EOrWei,MirSim} and conjecturally Gromov-Witten invariants of (local) toric Calabi-Yau 3-folds \cite{BKMP,ZhoLoc}.

The Eynard-Orantin invariants are defined via recursion relations that give an effective  algorithm to calculate $\omega^{g}_n(C)$ from $\omega^{g'}_{n'}(C)$ for $g'+n'\leq g+n$ and $g'\leq g$.  For example, $F^{(2)}(C)$ requires one to first calculate $\omega^{g}_{n}(C)$ for $(g,n)=(0,1),(0,2),(0,3),(1,1),(1,2)$ and $(2,1)$.  It is more difficult to find a non-recursive expression for the invariants.   The following theorem give a non-recursive expression for the invariants for the curve $xy-y^2=1$ in terms of the moduli space of curves.
\begin{theorem} \label{th:symp}
The Eynard-Orantin symplectic invariants of the curve $xy-y^2=1$ are the orbifold Euler characteristics of the moduli spaces of genus $g$ curves:
\[ F^{(g)}=\chi(\modm_g),\quad g>1.\]
\end{theorem}

The theorem is proven by associating the following Hurwitz problem to the invariants $\omega^{g}_n$ of $xy-y^2=1$.  Consider genus $g$ branched covers of $S^2$ unramified over $S^2-\{0,1,\infty\}$ with points in the fibre over $\infty$ labeled $(p_1,...,p_n)$ and with ramification $(b_1,...,b_n)$, ramification $(2,2,...,2)$ over $1$ and ramification greater than 1 at all points above $0$.  Define $N_{g,n}(b_1,...,b_n)\in\bq$ to be the weighted count of connected such coverings, counted up to isomorphism so that the weight of each branched cover is one divided by the order of its group of automorphisms.   This was studied in \cite{NorCou} where it was shown that $N_{g,n}(b_1,...,b_n)$ is a symmetric quasi-polynomial in the $b_i$ in the sense that it is polynomial on each coset of the sublattice of finite index $2\bz^n\subset\bz^n$, symmetric under permutations that leave a coset invariant.  Equivalently,
there exists polynomials $N_{g,n}^{(k)}(b_1,...,b_n)$ for $k=1,...,n$ such that $N_{g,n}(b_1,...,b_n)$ decomposes
\[ N_{g,n}(b_1,...,b_n)=N_{g,n}^{(k)}(b_1,...,b_n),\quad b_i{\rm\ odd\ }\ i\leq k,\ \ b_i{\rm\ even\ }\ i> k\]
and $N_{g,n}^{(k)}(b_1,...,b_n)$ is symmetric in $b_i$ for $i\leq k$ and in $b_i$ for $ i> k$.
In fact only even $k$ is necessary since by definition $N_{g,n}(b_1,...,b_n)$ vanishes if the number of odd $b_i$ is odd.  Each $N_{g,n}^{(k)}(b_1,...,b_n)$ is a polynomial in the $b_i^2$, symmetric under permutations that preserve the parity of the $b_i$.  The number $N_{g,n}(b_1,...,b_n)$ is presented in \cite{NorCou} in terms of counting lattice points inside integral convex polytopes depending on $(b_1,...,b_n)$ which make up a cell decomposition of $\modm_{g,n}$, the moduli space of genus $g$ curves with $n$ labeled points, and hence is said to count lattice points in the moduli space of curves.

The generating function 
\[F^{(g)}_n(z_1,...,z_n)=\sum_{b_i>0} N_{g,n}(b_1,...,b_n)z_1^{b_1}...z_n^{b_n}\] 
has radius of convergence of 1 in each variable, and extends to a meromorphic function in each variable on the whole complex plane.  See Lemma~\ref{th:inv} in Section~\ref{sec:rec}.

It was proven in \cite{NorCou} that the $N_{g,n}(b_1,...,b_n)$ satisfy recursion relations which uniquely determine them from $N_{0,3}(b_1,b_2,b_3)=1$ (when $b_1+b_2+b_3$ is even and zero otherwise.)
These recursion relations are used to prove the following theorem.

\begin{theorem}  \label{th:main}
For $2g-2+n>0$
\[\Omega^g_n=\frac{\partial}{\partial z_1}...\frac{\partial}{\partial z_n}F^{(g)}_ndz_1...dz_n\]
are the Eynard-Orantin invariants of the plane curve $xy-y^2=1$.
\end{theorem}
The Eynard-Orantin invariants satisfy further recursion relations known as string and dilaton equations---see Section~\ref{sec:string}.  They give rise to new recursion relations between the $N_{g,n}$ which we also call string and dilaton equations.  The first two of these are the string equations.
\begin{theorem}[String equations]   \label{th:string}
\[N_{g,n+1}(1,b_1,...,b_n)=\sum_{j=1}^n\sum_{k=1}^{b_j}kN_{g,n}(b_1,...,b_n)|_{b_j=k}\]
\[N_{g,n+1}(2,b_1,...,b_n)=\sum_{j=1}^n\sum_{k=1}^{b_j}kN_{g,n}(b_1,...,b_n)|_{b_j=k}-\frac{1}{2}\sum_{j=1}^nb_jN_{g,n}(b_1,...,b_n)\]
\end{theorem}
\noindent The string equations uniquely determine the genus 0 quasi-polynomials.

The count of branched covers $N_{g,n}(b_1,...,b_n)$ requires the $b_i$ to be positive integers.  Nevertheless the polynomials $N^{(k)}_{g,n}(b_1,...,b_n)$ can be evaluated at $b_i=0$ so we use them to define $N_{g,n}(b_1,...,b_n)$ when some of the $b_i=0$.   For example, 
\begin{equation}  \label{eq:euler}
N_{g,n}(0,0,...,0):=N^{(0)}_{g,n}(0,0,...,0)=\chi(\modm_{g,n})
\end{equation} 
for $\chi(\modm_{g,n})$ the (orbifold) Euler characteristic of $\modm_{g,n}$ was proven in \cite{NorCou}.   In this way the dilaton equation, below, is a relation between the quasi-polynomials defining $N_{g,n}$.  
\begin{theorem}[Dilaton equation]   \label{th:dilaton}
\[N_{g,n+1}(2,b_1,...,b_n)-N_{g,n+1}(0,b_1,...,b_n)=(2g-2+n)N_{g,n}(b_1,...,b_n)\]
\end{theorem}
\noindent Theorem~\ref{th:dilaton} together with (\ref{eq:euler}) generalises the relation between Euler characteristics
\[\chi(\modm_{g,n+1})=(2-2g-n)\chi(\modm_{g,n})\] 
which is a direct consequence of an exact sequence of mapping class groups.  See Section~\ref{sec:eval}.

The importance of the dilaton equation is that together with Theorem~\ref{th:main} it is used to prove Theorem~\ref{th:symp}.  The importance of the string and dilaton equations is that together they are used to give a counting problem interpretation of $N_{g,n}(b_1,...,b_n)$ when some, but not all, of the $b_i=0$.  This is crucial to extending the lattice point count to the compactified moduli space \cite{DNoCou}.  Theorems~\ref{th:string} and \ref{th:dilaton} follow from Theorem~\ref{th:main} and general properties of Eynard-Orantin invariants.  Purely combinatorial proofs of the string equations do exist whereas the dilaton equation cannot have a purely combinatorial proof---one cannot specify 0 ramification at a point above infinity so $N_{g,n+1}(0,b_1,...,b_n)$ is not {\em a priori} the solution of a counting problem (and likewise for $N_{g,n}(0,0,...,0)$.)   

Section~\ref{sec:back} describes $N_{g,n}(b_1,...,b_n)$ and $\omega^g_n(p_1,...,p_n)$.  Section~\ref{sec:rec} contains the proof of Theorem~\ref{th:main} and  Section~\ref{sec:string} contains the proofs of Theorems~\ref{th:string} and \ref{th:dilaton}. 
Section~\ref{sec:eval} contains vanishing results for $N_{g,n}(b_1,...,b_n)$ and the proof of Theorem~\ref{th:symp}.  Examples are given in Section~\ref{sec:example}.

{\em Acknowledgements.} The author would like to thank Bertrand Eynard, Nick Scott and the anonymous referee for useful comments.

\section{Background}  \label{sec:back}
In this section we give a short introduction to the two main ingredients of the paper---the quasi-polynomials $N_{g,n}(b_1,...,b_n)$, and the Eynard-Orantin invariants $\omega^g_n(p_1,...,p_n)$.

\subsection{Lattice point count.}

To enumerate branched covers $f:\Sigma\to\bp^1$ unramified over $\bp^1-\{0,1,\infty\}$ and satisfying the labeling and ramification conditions in the introduction, the main tool we use is the fatgraph, also known as ribbon graph or dessin d'enfant, given by $\Gamma=f^{-1}[0,1]\subset\Sigma$.  A {\em fatgraph} is a graph $\Gamma$ with vertices of valency $>2$ equipped with a cyclic ordering of edges at each vertex.  Equivalently a fatgraph is an isotopy class of embeddings of a graph into an orientable surface with complement a union of disks, so in particular it has genus and boundary.  The type $(g,n)$ of a fatgraph is its genus $g$ and number of boundary components $n$, which satisfy $2-2g-n=V(\Gamma)-E(\Gamma)$ where $V(\Gamma)$ and $E(\Gamma)$ are the number of vertices, respectively edges, of the graph $\Gamma$.  A {\em labeled} fatgraph has labeled boundary components.   The branched cover $f:\Sigma\to\bp^1$ is equivalent to its labeled fatgraph $\Gamma=f^{-1}[0,1]$ equipped with positive integer edge lengths.  Isomorphisms between labeled fatgraphs with positive integer edge lengths correspond to isomorphisms between (labeled) covers.  The length $b_k$ of a boundary component corresponds to the ramification of a point over $\infty$.

For a labeled fatgraph $\Gamma$ of type $(g,n)$ and $(b_1,...,b_n)\in\bz_+^n$ define $N_{\Gamma}(b_1,...,b_n)$ to be the number of ways of assigning positive integer lengths to the edges of $\Gamma$ such that the lengths around the labeled boundary components are $(b_1,...,b_n)$.  The discussion above shows that we can decompose the count of branched covers according to their labeled fatgraphs:
 \[ N_{g,n}(b_1,...,b_n)=\sum_{\Gamma\in \fat}\frac{1}{|Aut\ \Gamma|}N_{\Gamma}(b_1,...,b_n)\]
where the indexing set $\fat$ is the space of labeled fatgraphs of genus $g$ and $n$ boundary components.   (The automorphisms of $\Gamma$ act as isomorphisms between different assignments of positive integer edge lengths to $\Gamma$.)  It is proven in \cite{NorCou} that $N_{g,n}(b_1,...,b_n)$ is quasi-polynomial in the $b_i$ despite the fact that each $N_{\Gamma}(b_1,...,b_n)$ is only piecewise quasi-polynomial in the $b_i$.

Fatgraphs appear in another context.  Let $\modm_{g,n}$ be the moduli space of genus $g$ curves with $n$ labeled points.  The {\em decorated} moduli space $\modm_{g,n}\times\br^n_+$ equips the labeled points with positive numbers $(b_1,...,b_n)$ \cite{PenDec}.   It has a cell decomposition due to Penner, Harer, Mumford and Thurston 
\begin{equation}  \label{eq:cell}
\modm_{g,n}\times\br^n_+\cong\left(\bigsqcup_{\Gamma\in \fat}P_{\Gamma}\right)/\sim
\end{equation}
which is proven using the existence and uniqueness of meromorphic quadratic differentials with foliations having compact leaves, known as Strebel differentials which can be described via labeled fatgraphs with lengths on edges.  

The cell $P_\Gamma\cong\br_+^{E(\Gamma)}$ can be identified with all metrics on $\Gamma$, i.e. assign positive lengths on edges.   The gluing $\sim$ of cells in (\ref{eq:cell}) arises via identification of metrics on fatgraphs---when the length of an edge $l_E\to 0$ we identify this with the metric on the fatgraph with the edge $E$ contracted, and we also identify isometric metrics on fatgraphs (arising from isomorphisms of labeled fatgraphs.)  The fibre of the homeomorphism (\ref{eq:cell}) over a fixed $n$-tuple of positive numbers $(b_1,...,b_n)$ yields a space homeomorphic to $\modm_{g,n}$ decomposed into compact convex polytopes 
\[P_{\Gamma}(b_1,...,b_n)=\{{\bf x}\in\br_+^{E(\Gamma)}|A_{\Gamma}{\bf x}={\bf b}\}\]
where ${\bf b}=(b_1,...,b_n)$ and $A_{\Gamma}:\br^{E(\Gamma)}\to\br^n$ is the incidence matrix of $\Gamma$ that maps an edge to the sum of its two incident boundary components.  Equivalently $P_{\Gamma}(b_1,...,b_n)\subset P_{\Gamma}$ consists of all metrics on $\Gamma$ with boundary lengths $(b_1,...,b_n)$.  When the $b_i$ are positive integers the polytope $P_{\Gamma}(b_1,...,b_n)$ is a rational polytope (since $A_{\Gamma}$ has integer entries) which naturally contains the positive integer edge length fatgraphs corresponding to the branched covers described above.   In other words, $N_{\Gamma}(b_1,...,b_n)$ can be identified with the number of integer points in the convex polytope $P_{\Gamma}(b_1,...,b_n)$ and hence is referred to as a lattice point count.  

The top homogeneous degree terms of the polynomial $N_{g,n}^{(k)}$ ($k$ even) representing $N_{g,n}$ coincides with (2 times) Kontsevich's volume polynomial \cite{KonInt}
\[\displaystyle V_{g,n}(b_1,...,b_n)=\sum_{\Gamma\in \fat}\frac{1}{|Aut\ \Gamma|}V_{\Gamma}(b_1,...,b_n)\]
where $V_{\Gamma}(b_1,...,b_n)$ is the volume of $P_{\Gamma}(b_1,...,b_n)$ induced from the Euclidean volumes on $\br^{E(\Gamma)}$ and $\br^n$.  In particular the top homogeneous degree terms of the polynomials $N_{g,n}^{(k)}$ are independent of $k$ (for $k$ even.)\\

\noindent {\em Remark.} The identification of integer points in convex polytopes inside the moduli space with branched covers $f:\Sigma\to\bp^1$ unramified over $\bp^1-\{0,1,\infty\}$ and satisfying the labeling and ramification conditions in the introduction requires the deep results of existence and uniqueness of Strebel differentials.  The association of the branched cover with its fatgraph is more elementary and is all that is needed for the recursion (\ref{eq:rec1}) and hence for Theorems~\ref{th:main}, \ref{th:string} and \ref{th:dilaton}.   The branched covers are maps and have the advantage of generalising in ways that points in moduli space cannot.  Two directions of generalisation both involve counting extra maps to get counts consisting of $N_{g,n}(b_1,...,b_n)$ plus further positive terms.  In one direction, studied in \cite{EOrAlg}, one can drop the condition on points above $0$ i.e. allow ramification 1 there.  The count, $M_{g,n}(b_1,...,b_n)$, is no longer quasi-polynomial, but it is determined by $N_{g,n}(b_1,...,b_n)$.    In another direction, studied in \cite{DNoCou}, one can allow {\em stable} maps, i.e. allow the domain to be nodal, to get $\overline{N}_{g,n}(b_1,...,b_n)$ which is quasi-polynomial in the $b_i$.

\subsection{Eynard-Orantin invariants.}  \label{sec:EO}

Eynard and Orantin \cite{EOrInv} associate multidifferentials $\omega^g_n(C)$ to any Torelli marked Riemann surface $C$ equipped with two meromorphic functions $x$ and $y$ with the property that the zeros of $dx$ are simple and the map
\[ \begin{array}[b]{rcl} C&\to&\bc^2\\p&\mapsto& (x(p),y(p))\end{array}\]
is an immersion.  A {\em multidifferential} is a differential in each variable or equivalently $\omega^g_n(C)$ is a meromorphic section of $(T^*C)^{\otimes n}$ on $C^n=C\times C\times...\times C$. (We abuse notation and  write $T^*C$ for its pull-backs over $C^n$.)  A {\em Torelli marking} of $C$ is a choice of symplectic basis $\{a_i, b_i\}_{i=1,..,g}$ of the first homology group $H_1(\bar{C})$ of the compact closure $\bar{C}$ of $C$.

If $\omega$ is a meromorphic 1-form on $\overline{\bc}$ and analytic at $z_0$ then $\omega(z_0)$ can be expressed in terms of local information around the poles of $\omega$ using the Cauchy kernel as follows
\begin{align}  \label{eq:cauch}
\omega(z_0)&=f(z_0)dz_0=\res{z=z_0}\frac{f(z)dz}{z-z_0}dz_0
=\res{z=z_0}\frac{dz_0}{z-z_0}f(z)dz\\
&=\sum_{\alpha}\res{z=\alpha}\frac{dz_0}{z_0-z}\omega(z)\nonumber
\end{align}
where the sum is over all poles $\alpha$ of $\omega(z)$.  Similarly, the derivative of a meromorphic function $g$ on $\overline{\bc}$ can be expressed in terms of local information around the poles of $\alpha$ of $g$
\begin{equation*}
 g'(z)dz=\res{w=z}\frac{g(w)dwdz}{(w-z)^2}=-\sum_{\alpha}\res{w=\alpha}g(w)\frac{dwdz}{(w-z)^2}.
\end{equation*}

The expressions $\displaystyle\frac{dz_0}{z_0-z}\omega(z)$ and $\displaystyle\frac{dwdz}{(w-z)^2}$ are examples of {\em bidifferentials} which are meromorphic sections of $T^*\overline{\bc}\otimes T^*\overline{\bc}$ on $\overline{\bc}\times\overline{\bc}$.  The latter of these two generalises to a bidifferential on a Riemann surface $C$ of any genus defined as the meromorphic differential $\eta_w(z)dz$ unique up to scale which has a double pole at $w\in C$ and all $A$-periods vanishing.   The scale factor can be chosen so that $\eta_w(z)dz$ varies holomorphically in $w$, and transforms as a 1-form in $w$ and hence it is naturally expressed as the unique bidifferential on $C$ 
\[ B(w,z)=\eta_w(z)dwdz,\quad \oint_{A_i}B=0,\quad B(w,z)\sim\frac{dwdz}{(w-z)^2} {\rm\  near\ }w=z.\]  
It is symmetric in $w$ and $z$.  The bidifferential $B(w,z)$ is called the {\em Bergmann Kernel} in \cite{EOrInv} following \cite{TyuPer}.  It is called the fundamental normalised differential of the second kind on $C$ in \cite{FayThe}.  Recall that a meromorphic differential is {\em normalised} if its $A$-periods vanish and it is of the {\em second kind} if its residues vanish.

For every $(g,n)\in\bz^2$ with $g\geq 0$ and $n>0$ Eynard and Orantin \cite{EOrInv} define a multidifferential denoted by $\omega^g_n(p_1,...,p_n)$ for $p_i\in C$.  When $2g-2+n>0$, $\omega^g_n(p_1,...,p_n)$ is defined recursively in terms of local  information around the poles of $\omega^{(g')}_{n'}(p_1,...,p_{n'})$ for $2g'+2-n'<2g-2+n$.  This is closely related to (\ref{eq:cauch}) and its generalisation to any Riemann surface $C$ which expresses a normalised differential of the second kind in terms of local information around its poles using the kernel $\kappa_{z_0,p}(z)=\int_p^{z_0}B(z,z')$ on $C$.

For $2g-2+n>0$, the poles of $\omega^g_n(p_1,...,p_n)$ occur at the zeros of $dx$.  Since each zero $\alpha$ of $dx$ is simple, for any point $p\in C$ close to $\alpha$ there is a unique point $\hat{p}\neq p$ close to $\alpha$ such that $x(\hat{p})=x(p)$.  The recursive definition of $\omega^g_n(p_1,...,p_n)$ uses only local information around zeros of $dx$ and makes use of the well-defined map $p\mapsto\hat{p}$ there.

Set $\omega^{(0)}_1=ydx$ (which agrees with the convention in \cite{EOrTop} but disagrees with the convention in \cite{EOrInv}.)

\begin{equation} \label{eq:berg}
\omega^{(0)}_2=B(p_1,p_2)
\end{equation}
For $2g-2+n>0$,
\begin{equation}  \label{eq:EOrec}
\omega^g_{n+1}(p_1,p_S)=\sum_{\alpha}\hspace{-2mm}\res{p=\alpha}K(p_1,p)\hspace{-.5mm}\biggr[\omega^{(g-1)}_{n+2}(p,\hat{p},p_S)+\hspace{-5mm}\displaystyle\sum_{\begin{array}{c}_{g_1+g_2=g}\\_{I\sqcup J=S}\end{array}}\hspace{-5mm}
\omega^{(g_1)}_{|I|+1}(p,p_I)\omega^{(g_2)}_{|J|+1}(\hat{p},p_J)\biggr]
\end{equation}
where the sum is over zeros $\alpha$ of $dx$, $I=\{i_i,...,i_k\}\subset S=\{2,...,n+1\}$ and $J$ are non-empty, $p_I=(p_{i_1},...,p_{i_k})$ (where its use is independent of the order of elements in $I$) and 
\begin{equation}   \label{eq:kernel} 
K(p_1,p)=\frac{-\int^p_{\hat{p}}B(p_1,p')}{2(y(p)-y(\hat{p}))dx(p)}
\end{equation} 
is well-defined in the vicinity of each zero of $dx$.   Note that the quotient of a differential by the differential $dx(z)$ thought of as sections of the canonical line bundle is a meromorphic function.  The poles of $\omega^g_n(p_1,...,p_n)$ occur at the zeros of of $dx$, and they are of order $6g-4+2n$.

The recursion (\ref{eq:EOrec}) depends only on the meromorphic differential $ydx$ and the map $p\mapsto\hat{p}$ around zeros of $dx$.  The simplest example of a plane curve with non-trivial Eynard-Orantin invariants is $y^2=x$.  It is known as the Airy curve since the Eynard-Orantin invariants reproduce Kontsevich's generating function \cite{KonInt} for intersection numbers on the moduli space.  

The simplicity of the curve $y^2=x$ can be measured by the divisor of its differential $ydx$ which is $(ydx)=2(0)-4(\infty)$.  The plane curve $xy-y^2=1$ also has extremely simple divisor $(ydx)=(-1)+(1)-(0)-3(\infty)$ and is the focus of this paper.

\section{Recursion}  \label{sec:rec}
In \cite{NorCou} the quasi-polynomials $N_{g,n}(b_1,...,b_n)$ were shown to satisfy the following recursion which uniquely determines $N_{g,n}$ from $N_{0,3}$ and $N_{1,1}$.
\begin{multline}   \label{eq:rec1}
b_1N_{g,n+1}(b_1,b_S)=\sum_{j>1}\frac{1}{2}\left[\sum_{p+q=b_1+b_j}pqN_{g,n}(b_S)|_{b_j=p}+\epsilon\hspace{-4mm}\sum_{p+q=|b_1-b_j|}\hspace{-4mm}pqN_{g,n}(b_S)|_{b_j=p}\right]\\
+\frac{1}{2}\sum_{p+q+r=b_1} pqr\biggl[N_{g-1,n+2}(p,q,b_S)+\hspace{-4mm}\sum_{\begin{array}{c}_{g_1+g_2=g}\\_{I\sqcup J=S}\end{array}}\hspace{-4mm}N_{g_1,|I|+1}(p,b_I)N_{g_2,|J|+1}(q,b_J)\biggr].
\end{multline}
where $\epsilon={\rm sgn}(b_1-b_j)$ and $b_I=(b_{i_1},...,b_{i_k})$ for $I=\{i_i,...,i_k\}\subset S=\{2,...,n+1\}$.
 (A factor of $1/2$ is incorrectly missing from the formula in \cite{NorCou}.  I am indebted to Norman Do and Motohico Mulase who independently pointed this out.)

\begin{lemma}
For ${\rm w}^g_n(z_1,...,z_n)=\frac{\partial}{\partial z_1}...\frac{\partial}{\partial z_n}F^{(g)}_n$ the recursion (\ref{eq:rec1}) is equivalent to
\begin{multline}   \label{eq:recgen}
\quad\quad{\rm w}^g_{n+1}(z,z_S)=\sum_{j=1}^n\frac{\partial}{\partial z_j}\left\{\left(\frac{z^3}{(1-z^2)^2}{\rm w}^g_{n}(z_S)|_{z_j=z}-\frac{z_j^3}{(1-z_j^2)^2}{\rm w}^g_{n}(z_S)\right)\right.\\
\hspace{14mm}\times\left.\left(\frac{1}{z-z_j}+\frac{z_j}{1-zz_j}\right)\right\}\\
+\frac{z^3}{(1-z^2)^2}\left[{\rm w}^{(g-1)}_{n+2}(z,z,z_S)\right.+\hspace{-4mm}\sum_{\begin{array}{c}_{g_1+g_2=g}\\_{I\sqcup J=S}\end{array}}\hspace{-4mm}
\left.{\rm w}^{(g_1)}_{|I|+1}(z,z_I){\rm w}^{(g_2)}_{|J|+1}(z,z_J)
\right]
\end{multline}
\end{lemma}
\begin{proof}
Apply the operator
\[P=\sum_{b_1=1}^{\infty}z^{b_1-1}\prod_{i=2}^{n+1}\sum_{b_i=1}^{\infty}b_iz_i^{b_i-1}\]  
to both sides of (\ref{eq:rec1}). The left hand side transforms to ${\rm w}^g_{n+1}(z_1,...,z_{n+1})$.
For each $j=1,...,n+1$, define the operator 
\[ P_j=\prod_{\begin{array}{c}1\leq i\leq n+1\\ i\neq j\end{array}}\sum_{b_i=1}^{\infty}b_iz_i^{b_i-1}.\]  
Then the $j$th summand in the first term on the right hand side of (\ref{eq:rec1}) transforms under $P$ to
\[P_j\sum_{b_1=1}^{\infty}z^{b_1-1}\sum_{b_j=1}^{\infty}b_jz_j^{b_j-1}
\sum_{p+q=b_1+b_j}\frac{1}{2}pqN_{g,n}(b_S)|_{b_j=p}\hspace{24mm}\]
\begin{align*}&=P_j\frac{\partial}{\partial z_j}\sum_{\tiny
{\begin{array}{c}p,q\\ q{\rm\ even}\end{array}}}\hspace{-2mm}\frac{1}{2}pqN_{g,n}(b_S)|_{b_j=p}\frac{z^{p+q}-z_j^{p+q}}{z-z_j}\\
&=P_j\frac{\partial}{\partial z_j}\sum_ppN_{g,n}(b_S)|_{b_j=p}\left[\frac{z^{p+2}}{(1-z^2)^2}-\frac{z_j^{p+2}}{(1-z_j^2)^2}\right]\frac{1}{z-z_j}\\
&=\frac{\partial}{\partial z_j}\left(\frac{z^3}{(1-z^2)^2}{\rm w}^g_{n}(z_S)|_{z_j=z}-\frac{z_j^3}{(1-z_j^2)^2}{\rm w}^g_{n}(z_S)\right)
\frac{1}{z-z_j}.
\end{align*}
In the second line above, the sum is over even $q$ because this is already the case in (\ref{eq:rec1}) since the summands vanish when $q$ is odd.  The restriction $p+q=b_1+b_j$ is lifted by summing over all $b_j$, and the quotient $(z^{p+q}-z_j^{p+q})/(z-z_j)$ neatly encodes the sum of all monomials $z^az_j^b$ with $a+b=p+q-1$.  In the third line, we have simply replaced 
\[ \sum_{{\rm even\ }q>0}\frac{1}{2}qz^q=\frac{z^2}{(1-z^2)^2}\] and similarly for $z_j$.

The transform under $P$ of the $j$th summand of the second term on the right hand side of (\ref{eq:rec1}) breaks into two sums
\[P_j\sum_{b_1=1}^{\infty}z^{b_1-1}\left\{\sum_{b_j=0}^{b_1}\ -\sum_{b_j=b_1+1}^{\infty}\right\}\  b_jz_j^{b_j-1}\hspace{-4mm}
\sum_{p+q=|b_j-b_1|}\frac{1}{2}pqN_{g,n}(b_S)|_{b_j=p}\hspace{12mm}\]
\begin{align*}&=P_j\frac{\partial}{\partial z_j}\ z_j\hspace{-3mm}\sum_{\tiny
{\begin{array}{c}p,q\\ q{\rm\ even}\end{array}}}\hspace{-2mm}\frac{1}{2}pqN_{g,n}(b_S)|_{b_j=p}\frac{z^{p+q}-z_j^{p+q}}{1-zz_j}\\
&=P_j\frac{\partial}{\partial z_j}\sum_ppN_{g,n}(b_1,..,b_n)|_{b_j=p}\left[\frac{z^{p+2}}{(1-z^2)^2}-\frac{z_j^{p+2}}{(1-z_j^2)^2}\right]\frac{z_j}{1-zz_j}\\
&=\frac{\partial}{\partial z_j}\left(\frac{z^3}{(1-z^2)^2}{\rm w}^g_{n}(z_S)|_{z_j=z}-\frac{z_j^3}{(1-z_j^2)^2}{\rm w}^g_{n}(z_S)\right)
\frac{z_j}{1-zz_j}.
\end{align*}
The third and fourth terms of the right hand side of (\ref{eq:rec1}) transform under $P$ to
\[
P_1\sum_{b_1=1}^{\infty}z^{b_1-1}\hspace{-4mm}\sum_{p+q+r=b_1}\hspace{-2mm} \frac{pqr}{2}\biggr[N_{g-1,n+2}(p,q,b_S)+\hspace{-5mm}\sum_{\begin{array}{c}_{g_1+g_2=g}\\_{I\sqcup J=S}\end{array}}\hspace{-4mm}N_{g_1,|I|+1}(p,b_I)N_{g_2,|J|+1}(q,b_J)\biggr]\quad\quad\]
\begin{align*}
&=P_1\sum_{r{\rm\ even}}\frac{1}{2}rz^{r+1}\sum_{p,q}pq\biggr[N_{g-1,n+2}(p,q,b_S)\\
&\hspace{38mm}+\hspace{-4mm}\sum_{\begin{array}{c}_{g_1+g_2=g}\\_{I\sqcup J=S}\end{array}}\hspace{-4mm}N_{g_1,|I|+1}(p,b_I)N_{g_2,|J|+1}(q,b_J)\biggr]z^{p+q-2}\\
&=P_1\frac{z^3}{(1-z^2)^2}\sum_{p,q}pq\biggr[N_{g-1,n+2}(p,q,b_S)\\
&\hspace{36mm}+\hspace{-4mm}\sum_{\begin{array}{c}_{g_1+g_2=g}\\_{I\sqcup J=S}\end{array}}\hspace{-4mm}N_{g_1,|I|+1}(p,b_I)N_{g_2,|J|+1}(q,b_J)\biggr]z^{p+q-2}\\
&=\frac{z^3}{(1-z^2)^2}\left[{\rm w}^{(g-1)}_{n+2}(z,z,z_S)\right.+\hspace{-4mm}\sum_{\begin{array}{c}_{g_1+g_2=g}\\_{I\sqcup J=S}\end{array}}\hspace{-4mm}
\left.{\rm w}^{(g_1)}_{|I|+1}(z,z_I){\rm w}^{(g_2)}_{|J|+1}(z,z_J)
\right].
\end{align*}
Thus the Lemma is proven.
\end{proof}
The meromorphic form $\Omega^g_n(z_1,...,z_n)$ is defined via its Taylor expansion around $z_j=0$, $j=1,...,n$ with radius of convergence 1.  The following lemma gives an explicit analytic continuation of $\Omega^g_n(z_1,...,z_n)$ to $|z_j|>1$.
\begin{lemma}  \label{th:inv}
\[\Omega^g_n(z_1,...,1/z_j,...,z_n)=-\Omega^g_n(z_1,...,z_j,...,z_n)\]
\end{lemma}
\begin{proof}
If $p(n)=\sum_{j=0}^kp_jn^j$ is a polynomial then
\[\sum_{n>0}p(n)z^n=\sum_{j=0}^kp_j\sum_{n>0}n^jz^n=\sum_{j=0}^kp_j\left(z\frac{d}{dz}\right)^j\hspace{-2mm}\frac{z}{1-z}\]
is an expansion around $z=0$ of a holomorphic function with radius of convergence 1 which follows from the convergence of $z+z^2+...$ for $|z|<1$.

If we restrict the parity of $n$ then
\[
f_+(z)=\sum_{\begin{array}{c}n>0\\n{\rm\ even}\end{array}}\hspace{-3mm}p(n)z^n=\sum_{j=0}^kp_j\left(z\frac{d}{dz}\right)^j\hspace{-2mm}\frac{z^2}{1-z^2}\]
\[f_-(z)=\sum_{\begin{array}{c}n>0\\n{\rm\ odd}\end{array}}\hspace{-3mm}p(n)z^n=\sum_{j=0}^kp_j\left(z\frac{d}{dz}\right)^j\hspace{-2mm}\frac{z}{1-z^2}
\]
are meromorphic functions with poles at $z=\pm 1$.  If we further consider only polynomials in $n^2$, $p(n)=\sum_{j=0}^kp_{2j}n^{2j}$ then the extension to $|z|>1$ is explicitly given by 
\[ f_+(z)+f_+(1/z)=-p(0),\quad f_-(z)+f_-(1/z)=0\]
which follows immediately from
\[\frac{z^2}{1-z^2}+\frac{1/z^2}{1-1/z^2}=-1,\quad \frac{z}{1-z^2}+\frac{1/z}{1-1/z^2}=0\]
and
\[ \left(z\frac{d}{dz}\right)^{2j}=\left(w\frac{d}{dw}\right)^{2j},\quad w=1/z.\]
Hence
\[F^{(g)}_n(z_1,...,1/z_j,...,z_n)=-F^{(g)}_n(z_1,...,z_j,...,z_n)+c(z_1,...,\hat{z}_j,...,z_n)\]
where $\frac{\partial}{\partial z_j}c(z_1,...,\hat{z}_j,...,z_n)=0$.
Thus $\Omega^g_n=\frac{\partial}{\partial z_1}...\frac{\partial}{\partial z_n}F^{(g)}_ndz_1...dz_n$ satisfies 
\[ \Omega^g_n(z_1,...,1/z_j,...,z_n)=-\Omega^g_n(z_1,...,z_j,...,z_n).\]

\end{proof}

\begin{proof}[Proof of Theorem~\ref{th:main}.]
Rewrite (\ref{eq:recgen}) as follows
\begin{multline}   \label{eq:recgen1}
{\rm w}^g_{n+1}(z,z_S)=\frac{z^3}{(1-z^2)^2}\Biggr\{\sum_{j=1}^n\left[\frac{1}{(z-z_j)^2}+\frac{1}{(1-zz_j)^2}\right]{\rm w}^g_{n}(z_S)|_{z_j=z}\\
\hspace{40mm}+{\rm w}^{(g-1)}_{n+2}(z,z,z_S)+\hspace{-4mm}\sum_{\begin{array}{c}_{g_1+g_2=g}\\_{I\sqcup J=S}\end{array}}\hspace{-4mm}
{\rm w}^{(g_1)}_{|I|+1}(z,z_I){\rm w}^{(g_2)}_{|J|+1}(z,z_J)\Biggr\}\\
\quad\quad-\sum_{j=1}^n\frac{\partial}{\partial z_j}\frac{z_j^3}{(1-z_j^2)^2}{\rm w}^g_{n}(z_S)\left(\frac{1}{z-z_j}+\frac{z_j}{1-zz_j}\right)
\end{multline}
and note that the last term is analytic at $z=\pm 1$.
In terms of $\Omega^g_{n+1}(z,z_S)={\rm w}^g_{n+1}(z,z_S)dz_Sdz$, (\ref{eq:recgen1}) becomes
\begin{multline}   \label{eq:recgen2}
\Omega^g_{n+1}(z,z_S)=\frac{1}{(z-\frac{1}{z})dx(z)}\Biggr\{\sum_{j=1}^n\left[\frac{dzdz_j}{(z-z_j)^2}+\frac{dzdz_j}{(1-zz_j)^2}\right]\Omega^g_{n}(z_S)|_{z_j=z}\\
\hspace{40mm}+\Omega^{(g-1)}_{n+2}(z,z,z_S)+\hspace{-7mm}\sum_{\begin{array}{c}_{g_1+g_2=g}\\_{I\sqcup J=S}\\_{(g_1,|I|)\neq(0,1)}\\_{(g_2,|J|)\neq(0,1)}\end{array}}\hspace{-7mm}
\Omega^{(g_1)}_{|I|+1}(z,z_I)\Omega^{(g_2)}_{|J|+1}(z,z_J)\Biggr\}\\
\quad\quad-\sum_{j=1}^n\partial_{z_j}\frac{1}{(z_j-1/z_j)dx(z_j)}\Omega^g_{n}(z_S)\left(\frac{dz}{z-z_j}+\frac{z_jdz}{1-zz_j}\right)
\end{multline}
where various differentials have necessarily appeared on the right hand side and $\partial_{z_j}=\frac{\partial}{\partial z_j}\{\cdot\}dz_j$.
The terms
\[\frac{dzdz_j}{(z-z_j)^2}=\Omega^{(0)}_2(z,z_j),\quad\quad\frac{dzdz_j}{(1-zz_j)^2}=-\frac{d(1/z)dz_j}{(1/z-z_j)^2}=-\Omega^{(0)}_2(1/z,z_j)\]
can be absorbed into the sum over $g_1+g_2=g$ to give
\begin{multline}   \label{eq:recgen3}
\Omega^g_{n+1}(z,z_S)=\frac{-1}{(z-\frac{1}{z})dx(z)}\Biggr[\Omega^{(g-1)}_{n+2}(z,1/z,z_S)+\hspace{-5mm}\sum_{\begin{array}{c}_{g_1+g_2=g}\\_{I\sqcup J=S}\end{array}}\hspace{-5mm}
\Omega^{(g_1)}_{|I|+1}(z,z_I)\Omega^{(g_2)}_{|J|+1}(1/z,z_J)\Biggr]\\
\hspace{12mm}-\sum_{j=1}^n\partial_{z_j}\frac{1}{(z_j-1/z_j)dx(z_j)}\Omega^g_{n}(z_S)\left(\frac{dz}{z-z_j}+\frac{z_jdz}{1-zz_j}\right)
\end{multline}
and we have also used $\Omega^{(g')}_{n'}(z,z_I)=-\Omega^{(g')}_{n'}(1/z,z_I)$.

Apply (\ref{eq:cauch}) to $\Omega^g_{n+1}(z_1,z_S)$ to get
\begin{equation} \label{eq:proof}
\Omega^g_{n+1}(z_1,z_S)
=\sum_{\alpha=\pm 1}\res{z=\alpha}\frac{dz_1}{z_1-z}\Omega^g_{n+1}(z,z_S)
\end{equation}
We will substitute  (\ref{eq:recgen3}) into the right hand side of (\ref{eq:proof}) but first note that the last term of  (\ref{eq:recgen3}) can be dropped since it is analytic at $z=\pm 1$ hence does not contribute to the right hand side of (\ref{eq:proof}).  
\begin{multline}   \label{eq:recgen4}
\Omega^g_{n+1}(z_1,z_S)=\sum_{\alpha=\pm 1}\res{z=\alpha}\frac{-1}{(z-1/z)dx(z)}\frac{dz_1}{z_1-z}\Biggr\{\Omega^{(g-1)}_{n+2}(z,1/z,z_S)\\
+\hspace{-4mm}\sum_{\begin{array}{c}_{g_1+g_2=g}\\_{I\sqcup J=S}\end{array}}\hspace{-4mm}
\Omega^{(g_1)}_{|I|+1}(z,z_I)\Omega^{(g_2)}_{|J|+1}(1/z,z_J)\Biggr\}
\end{multline}
by symmetry under $z\mapsto 1/z$ 
\begin{multline}  \label{eq:recgen5}
\Omega^g_{n+1}(z_1,z_S)=\sum_{\alpha=\pm 1}\res{z=\alpha}\frac{1}{(z-1/z)dx(z)}\frac{dz_1}{z_1-1/z}\Biggr\{\Omega^{(g-1)}_{n+2}(z,1/z,z_S)\\
+\hspace{-4mm}\sum_{\begin{array}{c}_{g_1+g_2=g}\\_{I\sqcup J=S}\end{array}}\hspace{-4mm}
\Omega^{(g_1)}_{|I|+1}(z,z_I)\Omega^{(g_2)}_{|J|+1}(1/z,z_J)\Biggr\}.
\end{multline}
The recursion (\ref{eq:EOrec}) defining the Eynard-Orantin invariants for the curve $xy-y^2=1$ in terms of the parametrisation 
\[x(z)=z+1/z,\quad y(z)=z\] 
is given by
\[
\omega^g_{n+1}(z_1,z_S)\hspace{-.5mm}=\hspace{-2mm}\sum_{\alpha=\pm 1}\hspace{-2mm}\res{z=\alpha}K(z_1,z)\hspace{-.5mm}\biggr[\omega^{(g-1)}_{n+2}(z,1/z,z_S)+\hspace{-5mm}\displaystyle\sum_{\begin{array}{c}_{g_1+g_2=g}\\_{I\sqcup J=S}\end{array}}\hspace{-5mm}
\omega^{(g_1)}_{|I|+1}(z,z_I)\omega^{(g_2)}_{|J|+1}(1/z,z_J)\biggr]
\]
where $z_I=(z_{i_1},...,z_{i_k})$ for $I=\{i_i,...,i_k\}$ a non-empty subset of $S=\{2,...,n+1\}$ with non-empty complement, and we have used the fact that the zeros of $dx$ are $z=\pm 1$ and the map $z\mapsto\hat{z}=1/z$ is global.  

The kernel (\ref{eq:kernel}) for the curve $x(z)=z+1/z$, $y(z)=z$ is given by
\begin{align*}
 K(z_1,z)&=\frac{-\int^z_{\hat{z}}B(z_1,z')}{2(y(z)-y(\hat{z}))dx(z)}\\
&=\frac{-1}{2(z-1/z)dx(z)}\left(\frac{dz_1}{z_1-z}-\frac{dz_1}{z_1-1/z}\right)
\end{align*}
where we have used $B(w,z)=dwdz/(w-z)^2$ and $\hat{z}=1/z$.  This kernel appears in (\ref{eq:recgen4})+(\ref{eq:recgen5}) hence
\[
\Omega^g_{n+1}(z_1,z_S)\hspace{-.5mm}=\hspace{-2mm}\sum_{\alpha=\pm 1}\hspace{-2.5mm}\res{z=\alpha}\hspace{-.5mm}K(z_1,z)\hspace{-.5mm}\biggr[\Omega^{(g-1)}_{n+2}(z,1/z,z_S)+\hspace{-5mm}\displaystyle\sum_{\begin{array}{c}_{g_1+g_2=g}\\_{I\sqcup J=S}\end{array}}\hspace{-5mm}
\Omega^{(g_1)}_{|I|+1}(z,z_I)\Omega^{(g_2)}_{|J|+1}(1/z,z_J)\hspace{-.5mm}\biggr]
\]
which coincides with the recursion relation (\ref{eq:EOrec}) that defines $\omega^g_n$ for  $x(z)=z+1/z$, $y(z)=z$.  

It is easy to calculate $\Omega^{(0)}_3$ and $\omega^{(0)}_3$ since $N_{0,3}(b_1,b_2,b_3)$ is known (it equals 1 when $b_1+b_2+b_3$ is even and 0 when $b_1+b_2+b_3$ is odd.)  Hence
\[\Omega^{(0)}_3=\left\{\frac{1}{2\prod(1-z_i)^2}-\frac{1}{2\prod(1+z_i)^2}\right\}\prod dz_i=\omega^{(0)}_3\]
and together with the recursion relation this proves that
\[\Omega^g_n=\omega^g_n(C)\]
for $C=\{xy-y^2=1\}$ as required.
\end{proof}

\section{String and dilaton equations}  \label{sec:string}
The Eynard-Orantin invariants satisfy the following {\em string equations} \cite{EOrInv}.
\begin{equation}  \label{eq:string1}
\sum_{\alpha}\res{z=\alpha}y(z)\omega_{n+1}^g(z,z_S)=-\sum_{j=1}^ndz_j\frac{\partial}{\partial z_j}\left(\frac{\omega^g_n(z_S)}{dx(z_j)}\right)
\end{equation}
\begin{equation}  \label{eq:string2}
\sum_{\alpha}\res{z=\alpha}x(z)y(z)\omega_{n+1}^g(z,z_S)=-\sum_{j=1}^ndz_j\frac{\partial}{\partial z_j}\left(\frac{x(z_j)\omega^g_n(z_S)}{dx(z_j)}\right).
\end{equation}
where $z_S=(z_1,...,z_n)$ and the sum is over the zeros $\alpha$ of $dx$.
\subsection{String and dilaton equations for $N_{g,n}$.}  \label{sec:stringproof}
The string equations (\ref{eq:string1}) and (\ref{eq:string2})
transform to simple equations in the $N_{g,n}$.
\begin{proof}[Proof of Theorem~\ref{th:string}]
The quasi-polynomial
$N_{g,n+1}(1,b_1,...,b_n)$ is the coefficient of $\prod_{i=i}^nb_iz_i^{b_i-1}dz_i$ in the expansion of $\res{z=0}\displaystyle\frac{1}{z}\omega_{n+1}^g(z,z_S)$ around $(z,z_S)=0$.
\begin{align*}
\res{z=0}\frac{1}{z}\omega_{n+1}^g(z,z_S)
&=-\res{z=\infty}z\omega_{n+1}^g(z,z_S)\\
&=-\res{z=\infty}y(z)\omega_{n+1}^g(z,z_S)\\
&=\sum_{\alpha=\pm 1}\res{z=\alpha}y(z)\omega_{n+1}^g(z,z_S)\\
&=-\sum_{j=1}^ndz_j\frac{\partial}{\partial z_j}\left(\frac{\omega^g_n(z_S)}{dx(z_j)}\right)\\
&=\sum_{j=1}^n\frac{\partial}{\partial z_j}\left\{(z_j^2+z_j^4+z_j^6+...)\omega^g_n(z_S)\right\}
\end{align*}
The first equality uses $\omega_{n+1}^g(z,z_S)/z=-\omega_{n+1}^g(1/z,z_S)/z$.  The poles of the meromorphic form $y(z)\omega_{n+1}^g(z,z_S)$ occur at $z=-1,1,\infty$ so the third equality uses the zero sum over all residues.  The fourth equality is (\ref{eq:string1}).  We have expanded $1/dx(z_j)=1/(1-1/z_j^2)dz_j$ around $z_j=0$ to show that in the expansion of the final term around $z_S=0$, the coefficient of $\prod_{i=1}^nb_iz_i^{b_i-1}dz_i$ is the right hand side of the following
\begin{align*}
N_{g,n+1}(1,b_1,...,b_n)&=\sum_{j=1}^n\hspace{-.9cm}\sum_{\tiny{\begin{array}{c}k<b_j\\\quad\quad\quad k\equiv b_j+1({\rm mod}\ 2)\end{array}}}\hspace{-1cm}kN_{g,n}(b_1,...,b_n)|_{b_j=k}\\
&=\sum_{j=1}^n\sum_{k=1}^{b_j}kN_{g,n}(b_1,...,b_n)|_{b_j=k}
\end{align*}
where each summand with $k\equiv b_j({\rm mod}\ 2)$ vanishes since $k+\sum_{i\neq j}b_i$ is odd.  This proves the first recursion of Theorem~\ref{th:string}.\\

The quasi-polynomial 2$N_{g,n+1}(2,b_1,...,b_n)$ appears in the expansion around $z_S=0$ of $\res{z=0}\displaystyle\frac{1}{z^2}\omega_{n+1}^g(z,z_S)$ as the coefficient of $\prod_{i=1}^nb_iz_i^{b_i-1}dz_i$.
\begin{align*}
\res{z=0}\frac{1}{z^2}\omega_{n+1}^g(z,z_S)
&=-\res{z=\infty}z^2\omega_{n+1}^g(z,z_S)\\
&=-\res{z=\infty}x(z)y(z)\omega_{n+1}^g(z,z_S)\\
&=\sum_{\alpha=\pm 1}\res{z=\alpha}x(z)y(z)\omega_{n+1}^g(z,z_S)\\
&=-\sum_{j=1}^ndz_j\frac{\partial}{\partial z_j}\left(\frac{x(z_j)\omega^g_n(z_S)}{dx(z_j)}\right)\\
&=\sum_{j=1}^n\frac{\partial}{\partial z_j}\left\{(z_j+2z_j^3+2z_j^5+...)\omega^g_n(z_S)\right\}
\end{align*}
The first equality is as above.  The second equality replaces $z^2$ with $z^2+1=x(z)y(z)$ since $\omega_{n+1}^g(z,z_S)$ is analytic at $z=\infty$.  Again the poles of the meromorphic form $x(z)y(z)\omega_{n+1}^g(z,z_S)$ occur at $z=-1,1,\infty$ leading to the third equality. The fourth equality is (\ref{eq:string2}).  We have expanded $x(z_j)/dx(z_j)=(z_j-1/z_j)/(1-1/z_j^2)dz_j$ around $z_j=0$ to show that the coefficient of $\prod_{j=i}^nb_iz_i^{b_i-1}dz_i$ is the right hand side of 
\begin{align*}2N_{g,n+1}(2,b_1,...,b_n)&=2\sum_{j=1}^n\hspace{-1cm}\sum_{\tiny
{\begin{array}{c}k<b_j\\\quad\quad\quad k\equiv b_j({\rm mod}\ 2)\end{array}}}\hspace{-1cm}kN_{g,n}(b_1,...,b_n)|_{b_j=k}+\sum_{j=1}^nb_jN_{g,n}(b_1,...,b_n)\\
&=2\sum_{j=1}^n\sum_{k=1}^{b_j}kN_{g,n}(b_1,...,b_n)|_{b_j=k}-\sum_{j=1}^nb_jN_{g,n}(b_1,...,b_n)
\end{align*}
where each summand with $k\equiv b_j+1({\rm mod}\ 2)$ vanishes since $k+\sum_{i\neq j} b_i$ is odd.   This proves
the second recursion of Theorem~\ref{th:string}.
\end{proof}
\begin{cor}   \label{th:genus0}
The string equations determine the genus 0 invariants.
\end{cor}
\begin{proof}
The string equations determine the genus 0 invariants.  If $F(b_1^2,...,b_n^2)$ is a polynomial satisfying:
\begin{enumerate}
\item deg $F(b_1^2,...,b_n^2)=n-3$ 
\item $F(b_1^2,...,b_n^2)$ is symmetric in $b_1,...,b_k$
\item $F(b_1^2,...,b_n^2)$ is symmetric in $b_{k+1},...,b_n$
\item $F(1,b_2^2,...,b_n^2)=N_{g,n}^{(k)}(1,b_2,...,b_n)$
\item $F(b_1^2,...,b_{n-1}^2,2^2)=N_{g,n}^{(k)}(b_1,b_2,...,b_{n-1},2)$
\end{enumerate}
then $F(b_1^2,...,b_n^2)=N_{g,n}^{(k)}(b_1,b_2,...,b_n)$ since 
\begin{align*}
F(b_1^2,...,b_n^2)-N_{g,n}^{(k)}(b_1,b_2,...,b_n)&=(b_1^2-1)(b_n^2-4)G(b_1^2,...,b_n^2)\\
&=\prod_{i=1}^k(b_i^2-1)\prod_{j=k+1}^n(b_j^2-4)H(b_1^2,...,b_n^2)
\end{align*}
which must vanish identically to have degree $\leq n-3$.  Thus $N_{g,n}^{(k)}(b_1,b_2,...,b_n)$ is uniquely determined by the two string equations.  If $k=0$ or $n$ then the argument is similar, although only one of the string equations is needed.
\end{proof}

The Eynard-Orantin invariants also satisfy the {\em dilaton equation.}
\begin{equation}  \label{eq:dil}
\sum_{\alpha}\res{z=\alpha}\Phi(z)\omega_{n+1}^g(z,z_S)=(2g-2+n)\omega_{n}^g(z_S)
\end{equation}
where $d\Phi=ydx$ and the sum is over the zeros $\alpha$ of $dx$.  The function $\Phi$ is well-defined up to a constant in a neighbourhood of each zero of $dx$ and the left hand side of (\ref{eq:dil}) is independent of the choice of constant.
\begin{proof}[Proof of Theorem~\ref{th:dilaton}]
The coefficient of $\prod_{i=1}^nb_iz_i^{b_i-1}dz_i$ in the right hand side of (\ref{eq:dil}) is $(2g-2+n)N_{g,n}(b_1,...,.b_n)$.  The left hand side of (\ref{eq:dil}) becomes:
\begin{align*}
\sum_{\alpha=\pm 1}\res{z=\alpha}\Phi(z)\omega_{n+1}^g(z,z_S)&=-\sum_{\alpha=\pm 1}\res{z=\alpha}d\Phi(z)\int_0^z\omega_{n+1}^g(z',z_S)\\
&=-\sum_{\alpha=\pm 1}\res{z=\alpha}(z-\frac{1}{z})dz\int_0^z\omega_{n+1}^g(z',z_S)\\
&=\res{z=\infty}(z-\frac{1}{z})dz\int_0^z\omega_{n+1}^g(z',z_S)\\
&=-\res{z=\infty}\frac{z^2}{2}\omega_{n+1}^g(z,z_S)-\res{z=\infty}\frac{dz}{z}\int_0^z\omega_{n+1}^g(z',z_S)\\
&=\res{z=0}\frac{1}{2z^2}\omega_{n+1}^g(z,z_S)-\res{z=\infty}\frac{dz}{z}\int_0^z\omega_{n+1}^g(z',z_S)
\end{align*}
The identity $0={\rm Res\ }d(fg)={\rm Res\ }df.g+{\rm Res\ }f.dg$ applies to $f=\Phi$ and $g=\int_0^z\omega_{n+1}^g$ even though $\Phi$ is a multiply-defined function.  This yields the first equality.  At $z=0$ the integral $\int_0^z\omega_{n+1}^g(z',z_S)$ vanishes so $(z-1/z)\int_0^z\omega_{n+1}^g(z',z_S)$ is analytic at $z=0$ and has poles at $z=\pm1$ and $\infty$ which leads to the third equality.  In the final expression, $N_{g,n+1}(2,b_1,...,b_n)$ is the coefficient of $\prod_{i=1}^nb_iz_i^{b_i-1}dz_i$ in the expansion of $\res{z=0}\displaystyle\frac{1}{2z^2}\omega_{n+1}^g(z,z_S)$ around $z_S=0$.  At $z=\infty$, $\int_0^z\omega_{n+1}^g(z',z_S)$ is analytic, thus $\res{z=\infty}\frac{dz}{z}\int_0^z\omega_{n+1}^g(z',z_S)=\int_0^{\infty}\omega_{n+1}^g(z',z_S)$  which has coefficient of $\prod_{i=1}^nb_iz_i^{b_i-1}dz_i$ given by $N_{g,n+1}(0,b_1,...,b_n)$.  Hence
\[N_{g,n+1}(2,b_1,...,b_n)-N_{g,n+1}(0,b_1,...,b_n)=(2g-2+n)N_{g,n}(b_1,...,b_n).\]
\end{proof}
\noindent {\em Remarks.} 1. It was necessary in the proofs of Theorems~\ref{th:string} and \ref{th:dilaton} that the $b_i>0$.  The equations immediately extend to allow all $b_i$.  For example, the dilaton equation implies the relationship between the {\em polynomials}
\[N_{g,n+1}^{(k)}(2,b_1,...,b_n)-N_{g,n+1}^{(k)}(0,b_1,...,b_n)=(2g-2+n)N_{g,n}^{(k)}(b_1,...,b_n)\]
for $b_i>0$.  The left hand side and right hand side are polynomials that agree at infinitely many values in each variable hence they coincide and in particular allow $b_i=0$.  If $b_j=0$ in the string equation then the sum on the right hand side corresponding to $j$ is empty.  Similarly, the main recursion (\ref{eq:rec1}) restricts to the polynomial parts of $N_{g,n}$ and hence also allows $b_i=0$. 

2. Around the zero of $dx$ given by $z=1$ (and similarly for $z=-1$), the curve $C$ given by
\[ x(z)=z+1/z=2+(z-1)^2+O(z-1)^3,\quad y(z)=1+(z-1)\]
resembles the Airy curve 
\[C^{{\rm Airy}}=\{x=z^2,\ y=z\}\] 
due to the simple branching of $x$.  Eynard and Orantin proved \cite{EOrTop} that near a zero of $dx$, in this case $z_i\approx 1$, $i=1,...,n$, the asymptotic behaviour of $\omega^g_n(C)$ is described by $\omega^g_n(C^{{\rm Airy}})$.   The asymptotic behaviour of $\omega^g_n(z_1,...,z_n)$ is governed by the top degree terms of the quasi-polynomial $N_{g,n}(b_1,...,b_n)$.  Since $\omega^g_n(C^{{\rm Airy}})$ give generating functions for intersection numbers on $\overline{\modm}_{g,n}$ \cite{EOrInv} this can be used to prove that the coefficients  of the top degree terms of the quasi-polynomial $N_{g,n}(b_1,...,b_n)$ are intersection numbers on $\overline{\modm}_{g,n}$.  This was proven in a different way in \cite{NorCou} by using the fact that the lattice point count approximates Kontsevich's volume of the moduli space which has coefficients intersection numbers on $\overline{\modm}_{g,n}$.  

3.  Let  $\modm_{g,n}({\bf L})$ be the moduli space of connected oriented genus $g$ hyperbolic surfaces with $n$ labeled geodesic boundary components of non-negative real lengths $L_1,...,L_n$.  It comes equipped with a symplectic form which gives rise to the Weil-Petersson volume $V^{WP}_{g,n}({\bf L})$.  Mirzakhani proved that $V^{WP}_{g,n}({\bf L})$ is a polynomial in ${\bf L}=(L_1,...,L_n)$ \cite{MirSim}. Eynard and Orantin \cite{EOrWei} proved that for $2g-2+n>0$
\[\omega^g_n(C^{WP})=\frac{\partial}{\partial z_1}...\frac{\partial}{\partial z_n}\mathcal{L}\{V^{WP}_{g,n}\}dz_1...dz_n\]
are the Eynard-Orantin invariants of the plane curve 
\[C^{WP}=\{x=z^2,y=-\sin(2\pi z)/4\pi\}\] 
which strictly represents a sequence of algebraic curves obtained by truncating the expansion for $y$ around $z=0$.
The string and dilaton equations applied to the Weil-Petersson volumes \cite{DNoWei} are
\begin{align*}
 V^{WP}_{g,n+1}({\bf L},2\pi i)&=\sum_{k=1}^n\int_0^{L_k}L_kV^{WP}_{g,n}({\bf L})dL_k\\
\frac{\partial^2 V^{WP}_{g,n+1}}{\partial L_{n+1}^2}({\bf L},2\pi i)&=\mathcal{E}\cdot V^{WP}_{g,n}({\bf L})-(4g-4+n)V^{WP}_{g,n}({\bf L})\\
\frac{\partial V^{WP}_{g,n+1}}{\partial L_{n+1}}({\bf L},2\pi i)&=2\pi i(2g-2+n)V^{WP}_{g,n}({\bf L}).
\end{align*}
where $\mathcal{E}=\sum_{j=1}^nL_j\partial/\partial L_j$ is the Euler vector field.  They have a nice geometric interpretation.   Evaluation of  $V^{WP}_{g,n+1}$ at $L_{n+1}=i\theta$ is the volume of the moduli space of hyperbolic surfaces with a cone point of angle $\theta$ when $\theta<\pi$ (and related to the volume when $\theta>\pi$.)  The string and dilaton equations give some information about the moduli spaces as the cone point tends to $2\pi$.

The string and dilaton equations satisfied by $N_{g,n}$ and $V^{WP}_{g,n}$ are strikingly similar, particularly if one substitutes $L_k=2\pi i b_k$ and uses the analogy of discrete integration and differentiation.  This suggests that the volume polynomials $V^{WP}_{g,n}$ may satisfy further identities similar to those satisfied by $N_{g,n}$ such as the vanishing results.  The dilaton equation leads to the symplectic invariant $F^{(g)}$ of Eynard-Orantin and it is interesting that in both the cases $F^{(g)}(C^{WP})$ and $F^{(g)}(C)$ turn out to be an invariant of the classical moduli space $\modm_g$---its volume and Euler characteristic respectively.  This suggests that the general symplectic invariant $F^{(g)}$ of Eynard-Orantin is somehow related to $\modm_g$.

\subsection{Tau notation}
Let $c^{(k)}_{\bf m}$ be the coefficient of $b_1^{2m_1}...b_n^{2m_n}$ in $N_{g,n}^{(k)}$ where the $b_i$ have been ordered so that the first $k$ are odd and the others are even.  Define
\begin{equation}  \label{eq:tau}
\langle\tau^-_{m_1}...\tau^-_{m_k}\tau^+_{m_{k+1}}...\tau^+_{m_n}\rangle_{g,n}:=2^{2|{\bf m}|-g}{\bf m}!(3g-3+n-|{\bf m}|)!\times c^{(k)}_{\bf m}
\end{equation}
where $|{\bf m}|=\sum_1^nm_i$ and ${\bf m}!=\prod_1^nm_i!$.
Since $N_{g,n}^{(k)}$ is symmetric in its odd variables and its even variables this tau notation encodes the entire polynomial, and we allow the $\tau_j^{\pm}$ to be written in any order.  If $k$ is odd or $|{\bf m}|>3g-3+n$ then the bracket vanishes.

The tau notation follows Witten's tau notation for intersection numbers \cite{WitTwo}.   The coefficients of the polynomials $N_{g,n}^{(k)}$ may be intersection numbers.  In particular, it was proven in \cite{NorCou} that when $|m|=3g-3+n$ it is an intersection number, 
\[ \langle\tau^-_{m_1}...\tau^-_{m_k}\tau^+_{m_{k+1}}...\tau^+_{m_n}\rangle_{g,n}=\langle\tau_{m_1}...\tau_{m_n}\rangle_{g,n}=\int_{\overline{\modm}_{g,n}}c_1(L_1)^{m_1}...c_1(L_n)^{m_n}\]
(independent of even $k$) where $L_1,...,L_n$, are tautological line bundles over ${\overline{\modm}_{g,n}}$.  

Put  $s=3g-3+n-|{\bf m}|$, $\tau^-_{\bf m}=\tau^-_{m_1}...\tau^-_{m_k}$ and $\tau^+_{\bf m}=\tau^+_{m_{k+1}}...\tau^+_{m_n}$.  The string equations become
\[ \sum_{p=0}^{s+1}\binom{s+1}{p}2^{-2p}\langle\tau^-_p\tau_{\bf m}^-\tau_{\bf m}^+\rangle_{g,n+1}=\sum_{p=0}^{s+1}\binom{s+1}{p}\sum_{j=1}^nb_{p,m_j}\langle\tau_{m_1}^-...\tau^{\mp}_{m_j+p-1}...\tau^+_{m_n}\rangle_{g,n}\]
\[ \sum_{p=0}^{s+1}\binom{s+1}{p}\langle\tau^-_p\tau_{\bf m}^-\tau_{\bf m}^+\rangle_{g,n+1}=\sum_{p=0}^{s+1}\binom{s+1}{p}\sum_{j=1}^nb'_{p,m_j}\langle\tau_{m_1}^-...\tau^{\pm}_{m_j+p-1}...\tau^+_{m_n}\rangle_{g,n}\]
where $b_{p,m_j}$ and $b'_{p,m_j}$ depend only on $p$ and $m_j$ and $\tau^{\mp}_{m_j+p-1}$ reverses the parity---replace $\tau^-_j$ (respectively $\tau^+_j$) with $\tau^+_{m_j+p-1}$ (respectively $\tau^-_{m_j+p-1}$)---and $\tau^{\pm}_{m_j+p-1}$ keeps the parity the same.  When $s=-1$, both string equations reduce to the usual string equation for intersection numbers on the moduli space of curves \cite{WitTwo}.  The tau notation gives a constructive proof of Corollary~\ref{th:genus0}, which states that the string equations determine the genus zero invariants, since the system of equations is triangular in the genus 0 invariants.

The dilaton equation becomes
\[\frac{1}{s+1}\sum_{m_0=1}^{s+1}\binom{s+1}{m_0}\langle\tau^+_{m_0}\tau^-_{\bf m}\tau^+_{\bf m}\rangle_{g,n+1}=(2g-2+n)\langle\tau^-_{\bf m}\tau^+_{\bf m}\rangle_{g,n}.\]
When $s=-1$, the dilaton equation reduces to the usual dilaton equation for intersection numbers on the moduli space of curves \cite{WitTwo}.  It can be used to determine the genus 1 invariants.

\section{Evaluation at $b_j=0$.}   \label{sec:eval}

The count of branched covers $N_{g,n}(b_1,...,b_n)$ requires the $b_i$ to be positive integers since ramification 0 makes no sense.  We can define $N_{g,n}(b_1,...,b_n)$ for some $b_j=0$ by evaluation of its representing polynomial $N^{(k)}_{g,n}(b_1,...,b_n)$.  Using $N_{g,n}(b_1,...,b_n)$ with some $b_i=0$, one can define a compactified count of lattice points by compactifying the moduli space \cite{DNoCou}.  It gives rise to a polynomial with constant term the Euler characteristic of the compactified moduli space.  The dilaton equation from Theorem~\ref{th:dilaton} 
\[N_{g,n+1}(0,b_1,...,b_n)=N_{g,n+1}(2,b_1,...,b_n)-(2g-2+n)N_{g,n}(b_1,...,b_n)\]
enables us to make sense of evaluation at $b_j=0$ in terms of a counting problem.  Furthermore, as explained in the remark at the end of Section~\ref{sec:stringproof} the string and dilaton equations still hold when some $b_j=0$ and this enables us to prove vanishing results when some $b_j=0$.

\begin{lemma}[\cite{NorCou}]   \label{th:van}
If $\displaystyle\sum_{i=1}^nb_i\leq 4g-4+2n$ then $N_{g,n}(b_1,...,b_n)=0$ when all $b_i>0$.
\end{lemma}
\begin{proof}
If $N_{g,n}(b_1,...,b_n)>0$, there exists a degree $\sum b_i$ genus $g$ branched cover $\pi:C\to S^2$ branched over $0$, $1$ and $\infty$ with ramification $(b_1,...,b_n)$ over $\infty$ and ramification $(2,2,...,2)$ over 1.  By the Riemann-Hurwitz formula, 
\[\chi(\pi^{-1}(S^2-\{0,\infty\}))=-\frac{1}{2}\sum_{i=1}^nb_i.\]  
Thus
\[ 2-2g-n=\chi(\pi^{-1}(S^2-\{\infty\}))=-\frac{1}{2}\sum_{i=1}^nb_i+\#\pi^{-1}(0)>-\frac{1}{2}\sum_{i=1}^nb_i.\]
\end{proof}
Using the dilaton equation we can extend the vanishing result to allow some $b_j$ to be 0.
\begin{cor}   \label{th:van2}
If $0<\displaystyle\sum_{i=1}^nb_i\leq4g-4+2(n-p)$ then $N_{g,n}(b_1,...,b_n)=0$ where $p=\#\{b_i=0\}$. 
\end{cor}
\begin{proof}
The case $p=0$ is Lemma~\ref{th:van} and begins the inductive argument on $p$ where we allow any $n$.
Suppose $b_{n+1}=0$, hence $p=\#\{b_i=0\}\geq 1$ and $\{b_1,...,b_n\}$ contains $p-1$ zeros.   Assume 
\[0<\displaystyle\sum_{i=1}^{n+1}b_i\leq4g-4+2(n+1-p).\]   
On the right hand side of
\[N_{g,n+1}(0,b_1,...,b_n)=N_{g,n+1}(2,b_1,...,b_n)-(2g-2+n)N_{g,n}(b_1,...,b_n)\]
the first term vanishes by an inductive hypothesis since  $\#\{b_i=0\}=p-1$ and  
\[0<2+\displaystyle\sum_{i=1}^nb_i\leq4g-2+2(n+1-p)=4g-4+2(n+1-(p-1)).\]
The second term on the right hand side also vanishes by the inductive hypothesis since $\#\{b_i=0\}=p-1$ and
\[0<\displaystyle\sum_{i=1}^nb_i\leq4g-4+2(n+1-p)=4g-4+2(n-(p-1))\]
completing the induction.
\end{proof}
\noindent {\em Remarks.} 1. The vanishing result of Lemma~\ref{th:van} uses $\sum b_i\leq-2\chi$ where $\chi$ is the Euler characteristic of the cover.  If we try to interpret Corollary~\ref{th:van2} in a similar way, then we are led to the idea that the Euler characteristic should be $2-2g-(n-p)$ in place of $2-2g-n$ as if setting $p$ of the $b_i$ to be zero {\em removes $p$ punctures.}

2.  The Euler characteristic arguments of Lemma~\ref{th:van} and Corollary~\ref{th:van2} cannot detect connectedness of the cover suggesting there may be further vanishing results.  This is indeed the case for genus 0 shown in the following lemma.
\begin{lemma}  \label{th:van3}
If $0<\displaystyle\sum_{i=1}^nb_i\leq2(n-3)$ then $N_{0,n}(b_1,...,b_n)=0$.
\end{lemma}
\begin{proof}
This uses the main recursion relation (\ref{eq:rec1}) which becomes for $g=0$
\begin{multline*}   
b_0N_{0,n+1}(b_0,b_S)=\sum_{j>0}\frac{1}{2}\left[\sum_{p+q=b_0+b_j}pqN_{0,n}(b_S)|_{b_j=p}+\hspace{-4mm}\sum_{p+q=b_0-b_j}\hspace{-4mm}pqN_{0,n}(b_S)|_{b_j=p}\right]\\
+\frac{1}{2}\sum_{p+q+r=b_0} pqr\sum_{I\sqcup J=S}N_{0,|I|+1}(p,b_I)N_{0,|J|+1}(q,b_J)
\end{multline*}
where $S=\{1,...,n\}$ and $b_S=(b_1,...,b_n)$.  The recursion allows $b_i=0$ as explained in the remark at the end of Section~\ref{sec:stringproof}.  

We will prove the vanishing result by induction.  If $\sum b_i$ is odd then $N_{g,n}$ vanishes so we assume $\sum b_i$ is even.  When $n=4$, if $0<\sum b_i\leq 2$ then $(b_1,b_2,b_3,b_4)=(2,0,0,0)$ or $(1,1,0,0)$.  These can be explicitly evaluated using 
\[N_{0,4}^{(0)}(b_1,b_2,b_3,b_4)=-1+\frac{1}{4}\sum b_i^2,\quad N_{0,4}^{(2)}(b_1,b_2,b_3,b_4)=-\frac{1}{2}+\frac{1}{4}\sum b_i^2\]
to get $N_{0,4}^{(0)}(2,0,0,0)=0=N_{0,4}^{(2)}(1,1,0,0)$ as required.

Suppose $\sum_0^nb_i\leq 2(n-2)$ and choose $b_0$ to be the maximum of the $b_i$ so that we can easily interpret the second sum over $p+q=b_0-b_j$.  In the first summand the variables $(b_1,...,b_j=p,...,b_n)$ satisfy
\[\sum_{i=1}^nb_i-b_j+p=\sum_{i=0}^nb_i-b_0-b_j+p=\sum_{i=0}^nb_i-q\leq 2(n-2)-2=2(n-3)\]
where $q$ must be even and $q=0$ annihilates the summand through $pq$ so $q\geq 2$.  By the inductive assumption, $N_{0,n}(b_S)|_{b_j=p}$ vanishes since the sum of its variables is less than or equal to $2(n-3)$.  The variables of the second summand satisfy the same inequality by replacing the second $=$ with $\leq$ in the previous calculation since $-b_0-b_j+p\leq-b_0+b_j+p=q$ hence the second summand vanishes.  The variables of the third summand satisfy
\[ p+q+\sum_{i\in I}b_i+\sum_{i\in J}b_i=\sum_{i=0}^nb_i-r\leq 2(n-2)-2=2(n-3)\]
where as before $r\geq 2$.  Hence either 
\[p+\sum_{i\in I}b_i\leq2(|I|-2)\quad{\rm or}\quad q+\sum_{i\in J}b_i\leq2(|J|-2)\] 
so by the inductive assumption one of $N_{0,|I|+1}(p,b_I)$ and $N_{0,|J|+1}(q,b_J)$ vanishes so the summand vanishes and the induction is complete.
\end{proof} 
The vanishing result of Lemma~\ref{th:van3} is powerful enough to uniquely determine $N_{0,n}$ and we might expect to be able to write an explicit formula for each polynomial representing the quasi-polynomial $N_{0,n}$.  We have not succeeded in finding an explicit formula and instead we will be content with the following corollary.
\begin{cor}
\[N_{0,n}(b,0,...,0)=\prod_{k=1}^{n-3}\frac{b^2-4k^2}{4k}.\]
\end{cor}
\begin{proof}
Lemma~\ref{th:van3} implies that $N_{0,n}(b,0,...,0)=0$ for $b=2,...,2(n-3)$.  Thus $N_{0,n}(b,0,...,0)=c\prod_{k=1}^{n-3}(b^2-4k^2)/4k$ for some constant $c$, since $N_{0,n}(b,0,...,0)$ is a polynomial in $b^2$ of degree $n-3$.  Now $N_{0,n}(0,0,...,0)=\chi(\modm_{0,n})=(-1)^{n-1}(n-3)!$ hence $c=1$.
\end{proof}

\noindent {\em Remark.}  The relation
\begin{equation}  \label{eq:eulrec}
\chi(\modm_{g,n+1})=(2-2g-n)\chi(\modm_{g,n}){\rm\ for\ }2g-2+n>0
\end{equation} 
follows from the exact sequence of mapping class groups
\begin{equation}  \label{eq:exact} 
1\rightarrow\pi_1(C-\{p_1,...,p_n\})\rightarrow\Gamma_g^{n+1}\rightarrow\Gamma_g^n\rightarrow 1
\end{equation}
since the (orbifold) Euler characteristic is $\chi(\modm_{g,n})=\chi(\Gamma_g^n)$ and (\ref{eq:exact}) implies $\chi(\Gamma_g^n)=\chi(\Gamma_g^{n+1})/\chi(C-\{p_1,...,p_n\}$.  In particular \cite{HZaEul,PenPer},
\begin{align*}
&\chi(\modm_{g,n+1})=(-1)^n\frac{(2g-2+n)!}{(2g-2)!}\chi(\modm_{g,1}),\quad g>0\\ 
&\chi(\modm_{0,n+1})=(-1)^n(n-2)!\chi(\modm_{0,3}).
\end{align*}
For $n>0$, the dilaton equation specialises to (\ref{eq:eulrec}) by setting all the variables to 0:
\[ N_{g,n+1}(0,...,0)=N_{g,n+1}(2,...,0)-(2g-2+n)N_{g,n}(0,...,0)\]
but $N_{g,n+1}(2,...,0)=0$ by Lemma~\ref{th:van3} for $g=0$ and Corollary~\ref{th:van2} for $g>0$ and by (\ref{eq:euler}) $N_{g,n+1}(0,...,0)=\chi(\modm_{g,n+1})$, $N_{g,n}(0,...,0)=\chi(\modm_{g,n})$  so (\ref{eq:eulrec}) follows.  In some sense the dilaton equation reflects the exact sequence (\ref{eq:exact}).

\begin{proof}[Proof of Theorem~\ref{th:symp}]
The symplectic invariant $F^{(g)}$ of Eynard and Orantin for $g>1$ is defined by applying the dilaton equation to the case $n=0$ to get 
\[
\sum_{\alpha}\res{z=\alpha}\Phi(z)\omega_{1}^g(z)=:(2g-2)F^{(g)}.
\]
The proof of Theorem~\ref{th:dilaton} also applies to the $n=0$ case to give
\[ N_{g,1}(2)-N_{g,1}(0)=(2g-2)F^{(g)}.\]
But $N_{g,1}(2)=0$ by Lemma~\ref{th:van} and $N_{g,1}(0)=\chi(\modm_{g,1})$ by (\ref{eq:euler}).  Thus 
\[F^{(g)}=\frac{\chi(\modm_{g,1})}{2-2g}=\chi(\modm_g)\] 
where the second equality uses the $n=0$ case of (\ref{eq:eulrec}).
\end{proof}

\section{Examples}   \label{sec:example}
Here we give explicit formulae for the simplest Eynard-Orantin invariants $\omega^g_n$ for $(x,y)=(z+1/z,z)$ and the corresponding quasi-polynomials $N_{g,n}$.   The recursion relation (\ref{eq:EOrec}) begins with the kernels
\[ \omega^{(0)}_2(z_1,z_2)=\frac{dz_1dz_2}{(z_1-z_2)^2},\quad K(z_1,z)=\frac{1}{2}\frac{z^3}{(1-z^2)^2}\left(\frac{1}{z-z_1}-\frac{1}{1/z-z_1}\right)\frac{dz_1}{dz}\]
so
\begin{align*}
\omega^{(0)}_{3}(z_1,z_2,z_3)\hspace{-.5mm}&=\hspace{-1mm}\sum\hspace{-2mm}\res{z=\pm 1}\hspace{-2mm}K(z_1,z)\left(
\omega^{(0)}_{2}(z,z_2)\omega^{(0)}_{2}(1/z,z_3)+\omega^{(0)}_{2}(z,z_3)\omega^{(0)}_{2}(1/z,z_2)\right)\\
&=\left\{\frac{1}{2\prod(1-z_i)^2}-\frac{1}{2\prod(1+z_i)^2}\right\}\prod dz_i\\
\omega^{(1)}_{1}(z_1)&=\sum\res{z=\pm 1}K(z_1,z)\omega^{(0)}_{2}(z,1/z)\\
&=\left\{-\frac{1}{32}\frac{1-4z_1+z_1^2}{(1-z_1)^4}+\frac{1}{32}\frac{1+4z_1+z_1^2}{(1+z_1)^4}\right\}dz_1\\
&=\frac{z_1^3dz_1}{(1-z_1^2)^4}
\end{align*}
where in both cases the residue at $z=1$ contributes the summands with poles at $z_i=1$ and the residue at $z=-1$ contributes the summands with poles at $z_i=-1$.  The higher invariants are calculated similarly.
\begin{align*}
\omega^{(0)}_4&=\left\{\frac{3}{4\prod(1-z_i)^2}\sum\frac{z_i}{(1-z_i)^2}-\frac{3}{4\prod(1+z_i)^2}\sum\frac{z_i}{(1+z_i)^2}\right.\\
&\left.\quad+\frac{\sum z_iz_j(1+z_k^2)(1+z_l^2)}{2\prod(1-z_i^2)^2}\right\}\prod dz_i.\\
\omega^{(1)}_2&=\biggr\{\frac{5}{32\prod(1-z_i)^2}\sum\left(\frac{z_i^2}{(1-z_i)^4}-\frac{z_i}{4(1-z_i)^2}\right)+\frac{3z_1z_2}{32\prod(1-z_i)^4}\\
&\quad\quad+\frac{5}{32\prod(1+z_i)^2}\sum\left(\frac{z_i^2}{(1+z_i)^4}+\frac{z_i}{4(1+z_i)^2}\right)+\frac{3z_1z_2}{32\prod(1+z_i)^4}\\
&\quad\quad+\frac{z_1z_2}{8\prod(1-z_i^2)^2}\biggr\}dz_1dz_2\\
\omega^{(2)}_1&=\frac{21z^7(1+3z^2+z^4)dz}{(1-z^2)^{10}}
\end{align*}
In $\omega^{(0)}_4$ the sum over $\{i,j,k,l\}=\{0,1,2,3\}$ consists of 24 terms.

\begin{table}[h]  \label{tab:poly}
\begin{spacing}{1.4}  
\begin{tabular}{||l|c|c|c||} 
\hline\hline

{\bf g} &{\bf n}&\# odd $b_i$&$N_{g,n}(b_1,...,b_n)$\\ \hline

0&3&0,2&1\\ \hline
1&1&0&$\frac{1}{48}\left(b_1^2-4\right)$\\ \hline
0&4&0,4&$\frac{1}{4}\left(b_1^2+b_2^2+b_3^2+b_4^2-4\right)$\\ \hline
0&4&2&$\frac{1}{4}\left(b_1^2+b_2^2+b_3^2+b_4^2-2\right)$\\ \hline
1&2&0&$\frac{1}{384}\left(b_1^2+b_2^2-4\right)\left(b_1^2+b_2^2-8\right)$\\ \hline
1&2&2&$\frac{1}{384}\left(b_1^2+b_2^2-2\right)\left(b_1^2+b_2^2-10\right)$\\ \hline
2&1&0&$\frac{1}{2^{16}3^35}\left(b_1^2-4\right)\left(b_1^2-16\right)\left(b_1^2-36\right)\left(5b_1^2-32\right)$\\
\hline\hline
\end{tabular} 
\end{spacing}
\end{table}

The quasi-polynomials are a more compact way to express the $\omega^g_n$.  For example $N_{0,4}$ is the sum of five monomials whereas $\omega^{(0)}_4$ is the sum of 32 rational functions.  It may be useful to express Eynard-Orantin invariants of other curves more compactly.

\end{document}